\documentclass[useAMS,referee,usenatbib]{biom}
\usepackage{graphicx}
\usepackage{color}
\usepackage{amsmath}
\usepackage{amssymb}
\usepackage{graphicx}
\usepackage{subfigure}

\def\bSig\mathbf{\Sigma}

\usepackage{amsmath}

\title[AW-Fisher's method]{Properties of adaptively weighted Fisher's method}

\author{Yusi Fang$^{1,*}$\email{yuf31@pitt.edu},
Shaowu Tang$^{2,**}$\email{shaowu.tang@roche.com },
Zhiguang Huo$^{1,**}$\email{zhh18@pitt.edu} and
George C. Tseng$^{1,3,4,***}$\email{ctseng@pitt.edu},
Yongseok Park$^{1,****}$\email{yongpark@pitt.edu} \\
$^{1}$Department of Biostatistics, University of Pittsburgh, Pittsburgh, Pennsylvania, U.S.A. \\
$^{2}$Roche Molecular Systems, Inc \\
$^{3}$Department of Human Genetics, University of Pittsburgh, Pittsburgh, Pennsylvania, U.S.A. \\
$^{4}$Department of Computational \& Systems Biology, University of Pittsburgh, Pittsburgh, Pennsylvania, U.S.A.}

\begin{document}


\date{{\it Received October} 2007. {\it Revised February} 2008.  {\it
Accepted March} 2008.}



\pagerange{\pageref{firstpage}--\pageref{lastpage}}
\pubyear{2019}




\label{firstpage}


\begin{abstract}
Meta-analysis is a statistical method to combine results from multiple clinical or genomic studies with the same or similar research problems. It has been widely use to increase statistical power in finding clinical or genomic differences among different groups. One major category of meta-analysis is combining p-values from independent studies and the Fisher's method is one of the most commonly used statistical methods. However, due to heterogeneity of studies, particularly in the field of genomic research, with thousands of features such as genes to assess, researches often desire to discover this heterogeneous information when identify differentially expressed genomic features. To address this problem, \citet{Li2011adaptively} proposed very interesting statistical method, adaptively weighted (AW) Fisher's method,
where binary weights, $0$ or $1$, are assigned to the studies for each feature to distinguish potential zero or none-zero effect sizes. \citet{Li2011adaptively} has shown some good properties of AW fisher's method such as the admissibility. In this paper, we further explore some asymptotic properties of AW-Fisher's method including consistency of the adaptive weights and the asymptotic Bahadur optimality of the test.
\end{abstract}

%

\begin{keywords}
adaptive weights; Fisher's method; combining p-values; meta-analysis; consistency; asymptotic Bahadur optimality.
\end{keywords}


\maketitle


%

\section{Introduction}
Meta-analysis is one of the most commonly used statistical method to synthesize information from multiple studies, particularly when each single study does not have enough power to draw a meaningful conclusion due to weak signals or small effective sizes. There are two commonly used methods to combine results from different studies:
(1) directly combining the effect sizes and
(2) indirectly combining p-values from independent studies.
Because of heterogeneous nature of alternative distributions, slightly different study goals and types of data, in omics studies, combining p-values is often more appropriate and appealing. Commonly used p-value combing methods include the Fisher's method \citep{Fisher1925statistical},
the Stouffer's method \citep{Stouffer1949american}, the logit method \citep{lancaster1961combination}, and minimum p-value (min-P) and maximum p-value (max-P) methods \citep{Tippett1931methods,Wilkinson1951statistical}.

In addition, in omics studies, researchers are often more interested in identifying biomarkers that are differentially expressed (DE) with consistent patterns across multiple studies. However, most p-value combining methods such as Fisher's method are mainly targeting on the gain of statistical power without providing any further information about the heterogeneities of the expression patterns for detected biomarkers. This problem was first gained attention in functional magnetic resonance imaging (fMRI) research \citep{friston2005conjunction}, and many methods have been proposed to address this heterogeneity problem since then. For example, \citet{song2014hypothesis} proposed $r^{th}$ ordered p-value (rOP) method to test the alternative hypothesis in which signals exist in at least a given percentage of studies;
\cite{li2014meta} proposed a class of meta-analysis methods based on summaries of weighted ordered p-values (WOP); and \citet{Li2011adaptively} proposed an adaptively weighted (AW) Fisher's method for gene expression data, in which a binary weight, $0$ or $1$, is assigned to each study in order to distinguish the potential of existing group effects. Similar ideas such as AW-FEM and AW-Bayesian approach were applied to GWAS meta-analysis \citep{han2012interpreting, bhattacharjee2012subset}, where only the effect sizes in a subset of studies were assumed to be non-zero in alternative hypotheses \citep{flutre2013statistical}.

The AW-Fisher's method has appealing feature in practice. This is because additional information can be obtained though estimated adaptive weights for detected DE genes. The adaptive weight estimates $\hat{\textbf{w}}$ reflect a natural biological interpretation of whether or not a study contributes to the statistical significance of a gene on differentiating groups and provide a way for gene categorization in follow-up biological interpretations and explorations.

Here is a motivative example of the AW-Fisher's method. Figure \ref{Graph:motivativeExamp} shows the heatmaps of gene expressions for DE genes identified by Fisher's and AW-Fisher's methods for three tissue mouse datasets. Fisher's method does not provide any indication of contribution of studies to the statistical significance, while the adaptive weights of AW-Fisher's method can group together the genes that share the same gene expression pattern, therefore providing information of gene-specific heterogeneity. This information could be very appealing in genomic data analysis and potentially very useful to interpret biological mechanisms.

\begin{figure}[!ht]
\begin{center}
   \includegraphics[width=0.8\textwidth]{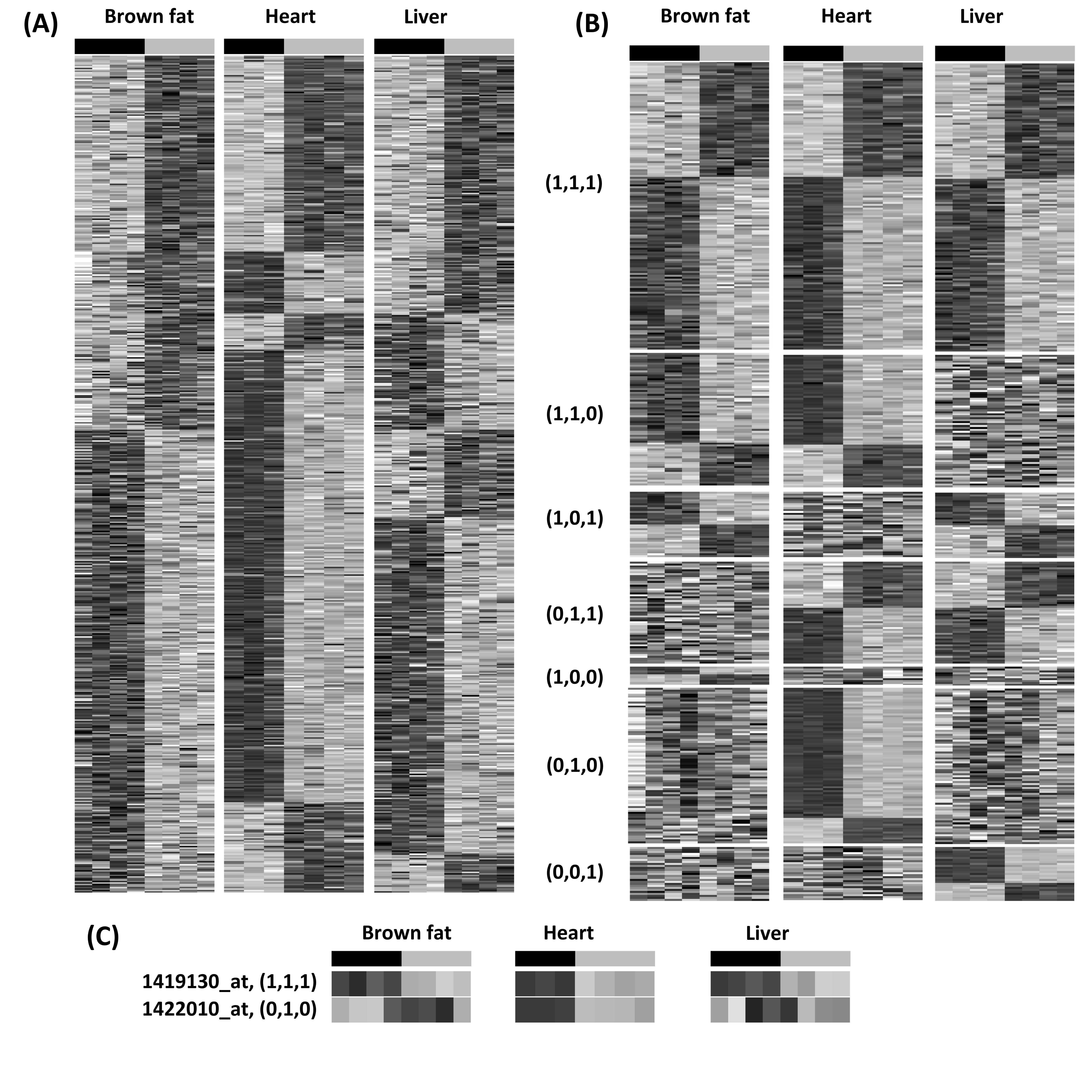}
\end{center}
\caption{Heatmaps of gene expressions for DE genes identified by Fisher's and AW-Fisher's methods in the mouse energy metabolism datasets.
(A) shows heatmap of gene expressions for DE genes identified by Fisher's method with false discovery rate 1\% (555 genes).
For each of the three tissues (Brown fat, Liver and Heart),
the group labels are on top of the heatmap, with black color represents wild type (WT)
and gray color represents Very long-chain acyl-CoA dehydrogenase (VLCAD) deficiency.
For the heatmap, darker color represent lower expression level and whiter color represent higher expression level.
(B) shows heatmap of gene expressions for DE genes identified by AW-Fisher's methods with false discovery rate 1\% (501 genes),
where only concordant genes are shown in the heatmap.
The AW weight categories are shown to the left of the heatmap.
(C) shows two specific genes from the mouse energy metabolism datasets.
Gene probe 1419130\_at belongs to $(1,1,1)$ category of AW-Fisher's methods.
Gene probe 1422010\_at belongs to $(0,1,0)$ category of AW-Fisher's methods.
}
\label{Graph:motivativeExamp}
\end{figure}

In this paper, we will further explore some asymptotic properties such as consistency of the adaptive weights and asymptotic Bahadur optimality (ABO) of the test \citep{bahadur1967rates}.

This paper is organized as follows. In Section 2, method will be briefly reviewed and in Section 3, consistency of AW-Fisher weights is addressed. Asymptotic Barhadur optimality will be discussed in Section 4. In section 5, simulations are used to show the consistency and exact slopes. The paper ends with discussion in Chapter 6.

\section{Adaptively weighted Fisher's method} \label{sec:AW-Fisher}
Considering to combine $K$ independent studies, denote the effect
size by $\vec{\boldsymbol{\theta}} = (\theta_1, \ldots, \theta_K) \in \mathbb{R}^K$ and let the corresponding p-value of study $k$ be $p_k$ for $k=1,\cdots,K$. The null and alternative hypothesis settings considered in this paper are
\begin{equation*}
 H_0:\ \vec{\boldsymbol{\theta}} \in \bigcap_{k=1}^K\{\theta_k=0\}\ \text{versus}\ H_A:\ \vec{\boldsymbol{\theta}} \in \bigcup_{k=1}^K\{\theta_k\neq 0\}.
\end{equation*}
Fisher's method summarizes p-values using the statistic  $T^{F}=-2\sum_{k=1}^K \log(p_k)$ to test this hypothesis setting. It has been shown that $T^{F}$ follows a $\chi^2$ distribution with $2K$ degrees of freedom if data from different studies are independent and there are no underlying difference from all studies ($\theta_k=0$ for all $k$). However, in genomic studies, usually tens of thousand of features are considered. The heterogeneous expression patterns are of interests. Fisher's method does not provide any information about the potential different expression patterns from different features.
\citet{Li2011adaptively} proposed an adaptively weighted Fisher's method to reveal this information through assigning binary weights.
Let vector $T(\vec{\textbf{w}}; \vec{\textbf{P}}) = -2 \sum_{k=1}^K w_k \log P_k$,
where $\vec{\textbf{w}} = (w_1, \ldots, w_K) \in {\{ 0,1 \} }^K$ is the AW weight associated with $K$ studies
and $\vec{\textbf{P}} = (P_1, \ldots, P_K) \in {(0,1)}^K$ is the random vector of  p-value vector for K studies.
Under the null distribution and conditional on $\vec{\textbf{w}}$,
the significance level using Fisher's method is
$L(T(\vec{\textbf{w}}; \vec{\textbf{P}})) = 1 - F_{\chi^2_{2d(\vec{\textbf{w}})}}(T(\vec{\textbf{w}}; \vec{\textbf{P}}))$,
where $d(\vec{\textbf{w}}) = 2\sum_{k=1}^Kw_k$ and
$F_{\chi^2_{2d(\vec{\textbf{w}})}}(\cdot)$ is the cumulative distribution function (CDF) of $\chi^2$-distribution
with degrees of freedom $2d(\vec{\textbf{w}})$.
The test statistic of AW-Fisher based on p-value vector $\vec{\textbf{P}}$ is defined as
$$s(\vec{\textbf{P}}) = -\log(\min_{\vec {\textbf{w}}} L(T(\vec{\textbf{w}}; \vec{\textbf{P}}))),$$
where optimal weight is determined by
\begin{equation}
\label{eq:weight}
\hat{ \textbf{w}} = \arg\min_{\vec{\textbf{w}}} L(T(\vec{\textbf{w}}; \vec{\textbf{p}})).
\end{equation}
Here $s(\cdot)$ is the mapping function from p-value vector to the AW-Fisher test statistic. Let $S = s(\vec{\textbf{P}})$, then equation (\ref{eq:weight}) implies that the best adaptive weights $\hat{ \textbf{w}}$ can be obtained by comparing all $2^K-1$ none-zero combinations of weights $\hat{ \textbf{w}}$.
In \citet{Li2011adaptively}, a permutation algorithm was proposed to calculate the p-value $P(S \ge  s_{obs})$,
where the observed  AW-statistic $s_{obs} = s(\vec{\textbf{p}}_{obs})$ and  $\vec{\textbf{p}}_{obs}$ is the observed p-values.
In Huo et al. (2017),
an importance sampling technique with spline interpolation and a linear weight search scheme is proposed to overcome computational burden and lack of accuracy for small p-values of the permutation algorithm.

\section{Consistency of the weight estimates}
\label{sec:Properties}

 Before we prove the consistency of AW-Fisher's weight estimates and the asymptotic Bahadur optimality of the AW-Fisher's method, we review exact slope and present some assumptions.

For $K$ independent studies to test $H_0: \theta_k=0$ with
sample size $n_k$ and p-value $p_k$ for $k=1,\cdots,K$, the statistical test for study $k$ has exact slope
$c_k(\theta)$, if
\begin{equation}
  -\frac{2}{n_k}\log(p_k)\rightarrow c_k(\theta)\ \mbox{as}\ n_k\rightarrow\infty.\nonumber
\end{equation}

The exact slope $c_k(\theta)$ is non-negative and used to measure how fast the p-value $p_k$ converge to zero
as $n_k$ goes to infinity. If the test statistics comes from alternative hypothesis, $c_k(\theta)$ is positive while under the null $c_k(\theta) = 0$.

When we consider the consistency, we assume the proportion of total samples assigned to each study are asymptotically fixed. I.e.
\begin{equation}
  \lim_{n\rightarrow\infty}\frac{n_k}{n}=\lambda_k\ \mbox{for}\ k=1,\cdots,K,\nonumber
\end{equation}
where $n=\frac{1}{K}\sum_{k=1}^K n_k$, the averaged sample size. Therefore \[-\frac{2}{n}\log(p_k)\rightarrow
\lambda_kc_k(\theta).\]

In addition, denote $100(1-\alpha)\%$
quantile of $\chi^2_m$ by $\chi^{-2}_m(\alpha)$, i.e.,
$P(\chi^2_{m}\geq \chi^{-2}_m(\alpha))=\alpha$. It can be seen that
$\chi^{-2}_m(\alpha)\rightarrow\infty$ as $\alpha\rightarrow 0$.

  Next we prove the main theorem of the consistency.

\begin{theorem} \label{thm:consistency}
Let ${\textbf{w}}^* = \{ \vec{\textbf{w}}: w_k = 1 \mbox{ if } \theta_k \ne 0 \text{ or } 0 \mbox{ if } \theta_k = 0\}$, the true weight vector.
$\hat{\textbf{w}}\rightarrow \textbf{w}^*$ as $n \rightarrow \infty$ where $\hat{\textbf{w}}$ statisfies Equation~(\ref{eq:weight}), i.e. asymptotically all and only the
studies with non-zero effect sizes will contribute to the AW statistic.
\end{theorem}

\begin{proof} Let
$$-\frac{2}{n}\sum_{j=1}^k \log(p_j) \rightarrow \sum_{j=1}^k \lambda_j c_j(\theta) = C_k$$
for $k=1,\dots,K,$ as $n \rightarrow \infty.$

Assume that $i$ studies have weight $1$ asymptotically. Without loss of generality, let first $i$ studies have weight $1$.

First, we prove that there is no study with $c(\theta) > 0$ with weight $0$. Suppose there exist $\ell$ such studies, say $(i+1)^{th},..., (i+\ell)^{th}$ such that $c_{i+1}(\theta),...,c_{i+\ell}(\theta) > 0$.
Denote $A_{i+\ell} = \frac{1-F_{\chi^{2}_{2i}}(-2\sum_{j=1}^i\log p_j)}{1-F_{\chi^{2}_{2(i+\ell)}}(-2\sum_{j=1}^{i+\ell}\log p_j)}$ and
$A_{i+\ell} < 1$ represents $\hat{w}_{i+1} = 0$ according to Equation~\ref{eq:weight}.
Since
\[
F_{\chi_{2i}^2}(t) = 1 - \sum_{j=0}^{i-1}\frac{1}{j!}\left(\frac{t}{2}\right)^j\exp\{-\frac{t}{2}\},
\]
we have
\[
\begin{split}
\lim_{n \rightarrow \infty} A_{i+\ell} &= \lim_{n \rightarrow \infty} \frac{1-F_{\chi^{2}_{2i}}(-2\sum_{j=1}^i\log p_j)}{1-F_{\chi^{2}_{2(i+\ell)}}(-2\sum_{j=1}^{i+\ell}\log p_j)}\\
& = \lim_{n \rightarrow \infty} \frac{1-F_{\chi^{2}_{2i}}(nC_i)}{1-F_{\chi^{2}_{2(i+\ell)}}(nC_{i+\ell})}\\
& = \lim_{n \rightarrow \infty} \frac{e^{-\frac{nC_i}{2}}\sum_{j=0}^{i-1}\frac{\left(\frac{nC_i}{2}\right)^j}{j!}}{e^{-\frac{nC_{i+\ell}}{2}}\sum_{j=0}^{i+\ell-1}\frac{\left(\frac{nC_{i+\ell}}{2}\right)^j}{j!}}\\
& = \lim_{n \rightarrow \infty} \frac{(i+\ell-1)!}{(i-1)!}\frac{C^{i-1}_i}{\left(\frac{n}{2}\right)^\ell C_{i+\ell}^{i+\ell-1}} (1+o(1)) \exp\left\{\frac{n}{2}(C_{i+\ell}-C_{i})\right\} \\
& \ge  \frac{1}{n^\ell} \frac{C^{i-1}_i}{ C_{i+\ell}^{i+\ell-1}}\exp\left\{\frac{n}{2}(C_{i+\ell}-C_{i})\right\} \\
& = \frac{1}{n^\ell} \frac{C^{i-1}_i}{ C_{i+\ell}^{i+\ell-1}}\exp\left\{\frac{n}{2}\left(\sum_{j=1}^{\ell}c_{i+j}\right)\right\}\\
& \rightarrow  \infty \\
\end{split}
\]

Then
\[
\begin{split}
&\lim_{n \rightarrow \infty} P(\hat{w}_{i+1} = 0,...,\hat{w}_{i+\ell}=0 | w^*_{i+1} = 1,...,w^*_{i+\ell} = 1)
= P(  \lim_{n \rightarrow \infty} A_{i+\ell} \le 1) \\
& \le  P\left(\lim_{n \rightarrow \infty }\frac{1}{n^\ell} \frac{C^{i-1}_i}{ C_{i+\ell}^{i+\ell-1}}\exp\left\{\frac{n}{2}(C_{i+\ell}-C_{i})\right\}\le1\right) = 0,
\end{split}
\]
i.e. the $\ell$ studies will eventually get a weight $1$ once $w^*_{i+1} \ne 0$.
The convergence rate is $O\left(n^\ell\exp\left\{-\frac{n}{2}\left(\sum_{j=1}^{\ell}c_{i+j}\right)\right\}\right)$.\\

Notice that here we only require that $C_{i+\ell}>0$ and $C_{i+\ell}>C_{i}$ to let the above argument hold.\\

Second, if there exist $\ell'$ studies with zero effect size that have a
weight $1$. Without loss of generality, let $\theta_{i-\ell'+1},...,\theta_i = 0$. In order to have
weight $1$ for these studies, one must have

\[
\begin{split}
P(\hat{w}_{i-\ell'+1} = 1,...,\hat{w}_{i}=1 | w^*_{i-\ell'+1} = 0,...,w^*_{i} = 0)  &
= P\left(  \frac{1-F_{\chi^{2}_{2i}}(-2\sum_{j=1}^i\log p_j)}{1-F_{\chi^{2}_{2(i-\ell')}}(-2\sum_{j=1}^{i-\ell'}\log p_j)} \le 1\right) \\
\end{split}
\]

Then we have

\[
\begin{split}
 &P(\hat{w}_{i-\ell'+1} = 1,...,\hat{w}_{i}=1 | w^*_{i-\ell'+1} = 0,...,w^*_{i} = 0)  \\ &
= P \left\{\sum_{j=0}^{i-1}\frac{1}{j!}\left(-\sum_{j=1}^i \log p_j\right)^j\exp\left(\sum_{j=1}^i \log p_j\right) \le \sum_{j=0}^{i-\ell'-1}\frac{1}{j!}\left(-\sum_{j=1}^{i-\ell'} \log p_j\right)^j\exp\left(\sum_{j=1}^{i-\ell'} \log p_j\right) \right\} \\
& = P \left(  \exp\left\{\sum_{j=i-\ell'+1}^i\log p_j\right\}\le\frac{\sum_{j=0}^{i-\ell'-1}\frac{1}{j!}\left(-\sum_{j=1}^{i-\ell'} \log p_j\right)^j}{\sum_{j=0}^{i-1}\frac{1}{j!}\left(-\sum_{j=1}^i \log p_j\right)^j} \right) \\
& \rightarrow P\left(\lim_{n \rightarrow \infty}\prod_{j=i-\ell'+1}^i p_i \le \left(\frac{n}{2}\right)^{-\ell'} \frac{(i-\ell'+1)!}{i!}C_{i-\ell'}^{i-\ell'-1}C_{i}^{-i+1}(1+o(1))\right) \\
& \rightarrow 0.
\end{split}
\]

Therefore, eventually no studies with zero effect size will have weight $1$ with convergence rate of $O(1/n^{\ell'})$.
From these two arguments we can see that only those with non-zero effect size will eventually be assigned to weight 1.
Note that the convergence rate for study with non-zero effective size to be weight 1 is
faster than that for study with zero effective size to be eventually assigned to 0.\\

 Let $\{p'_1, \dots, p'_j\}$ be a selected subset of $\{p_1, \dots, p_K\}$ with weights 1, while outside subset the weights are zero. Within the subset, assume $\{p'_1, \dots, p'_\ell\}$  have non-zero effect size, while the remainder have zero effect size. Further more, assume $\{p'_{j+1}, \dots, p'_{j+\ell'}\}$ outside the subset with true non-zero effect size. Denote $\{p_1, \dots, p_{\ell+\ell'}\}$ as the subset with the true non-zero effect size. \\

Based on the previous two results proved above, for $\ell>0$ ,then we have

 \[
\begin{split}
 &P(\hat{w}'_{1} = 1,...,\hat{w}'_{j}=1,\hat{w}'_{j+1}=0,...,\hat{w}'_{K}=0 | w^*_{1} = 1,...,w^*_{\ell+\ell'} = 1,w^*_{\ell+\ell'+1} = 0,...,w^*_{K} = 0)  \\ &
= P \left(\frac{1-F_{\chi^{2}_{2j}}(-2\sum_{i=1}^j\log p'_i)}{1-F_{\chi^{2}_{2(\ell+\ell')}}(-2\sum_{i=1}^{\ell+\ell'}\log p_i)} \le 1 \right) \\
& = P \left(  \frac{1-F_{\chi^{2}_{2j}}(-2\sum_{i=1}^j\log p'_i)}{1-F_{\chi^{2}_{2\ell}}(-2\sum_{i=1}^\ell\log p'_i)}\frac{1-F_{\chi^{2}_{2\ell}}(-2\sum_{i=1}^\ell\log p'_i)}{1-F_{\chi^{2}_{2(\ell+\ell')}}(-2\sum_{i=1}^{\ell+\ell'}\log p_i)} \le 1 \right) \\
& \le 2\left\{P\left(\frac{1-F_{\chi^{2}_{2j}}(-2\sum_{i=1}^j\log p'_i)}{1-F_{\chi^{2}_{2\ell}}(-2\sum_{i=1}^\ell\log p'_i)}\le1\right)+P\left(\frac{1-F_{\chi^{2}_{2\ell}}(-2\sum_{i=1}^\ell\log p'_i)}{1-F_{\chi^{2}_{2(\ell+\ell')}}(-2\sum_{i=1}^{\ell+\ell'}\log p_i)} \le 1\right) \right\}\\
& = 2\left\{P(\hat{w}'_{\ell+1} = 1,...,\hat{w}'_{j}=1| w'^*_{\ell+1} = 0,...,w'^*_{j} = 0)+P(\hat{w}'_{j+1} = 0,...,\hat{w}'_{j+\ell'}=0| w'^*_{j+1} = 1,...,w'^*_{j+\ell'} = 1) \right\}\\
& \rightarrow 0.
\end{split}
\]
as n goes to infinity.\\

For $\ell=0$, Then we have
\[
\begin{split}
 &P(\hat{w}'_{1} = 1,...,\hat{w}'_{j}=1,\hat{w}'_{j+1}=0,...,\hat{w}'_{K}=0 | w^*_{1} = 1,...,w^*_{\ell'} = 1,w^*_{\ell'+1} = 0,...,w^*_{K} = 0)  \\ &
= P \left(\frac{1-F_{\chi^{2}_{2j}}(-2\sum_{i=1}^j\log p'_i)}{1-F_{\chi^{2}_{2\ell'}}(-2\sum_{i=1}^{\ell'}\log p_i)} \le1 \right) \\
& = P \left(  \frac{1-F_{\chi^{2}_{2j}}(-2\sum_{i=1}^j\log p'_i)}{1-F_{\chi^{2}_{2(j+\ell')}}(-2\sum_{i=1}^{\ell'}\log p_i-2\sum_{i=1}^j\log p'_i)}\frac{1-F_{\chi^{2}_{2(j+\ell')}}(-2\sum_{i=1}^{\ell'}\log p_i-2\sum_{i=1}^j\log p'_i)}{1-F_{\chi^{2}_{2\ell'}}(-2\sum_{i=1}^{\ell'}\log p_i)} \le1 \right) \\
& \le 2P\left(\frac{1-F_{\chi^{2}_{2j}}(-2\sum_{i=1}^j\log p'_i)}{1-F_{\chi^{2}_{2(j+\ell')}}(-2\sum_{i=1}^{\ell'}\log p_i-2\sum_{i=1}^j\log p'_i)}\le1\right)\\
&+2P\left(\frac{1-F_{\chi^{2}_{2(j+\ell')}}(-2\sum_{i=1}^{\ell'}\log p_i-2\sum_{i=1}^j\log p'_i)}{1-F_{\chi^{2}_{2\ell'}}(-2\sum_{i=1}^{\ell'}\log p_i)} \le1\right) \\
& = 2\left\{P(\hat{w}_{1} = 0,...,\hat{w}_{\ell'}=0| w^*_{1} = 1,...,w^*_{\ell'} = 1)+P(\hat{w}'_{1} = 1,...,\hat{w}'_{j}=1| w'^*_{1} = 0,...,w'^*_{j} = 0) \right\}\\
& \rightarrow 0.
\end{split}
\]
as n goes to infinity.
\end{proof}

\section{The asymptotic Bahadur optimality of AW-Fisher's method}

\citet{bahadur1967rates} first discussed the asymptotic optimality of statistical test under the conditions of exact slopes and proportional sample sizes as shown in the previous section, which is called asymptotic Bahadur optimality (ABO) by using the ratio of the exact slopes of different statistical tests and the test
with larger exact slope is viewed as superior. \citet{littell1971asymptotic} showed that the Fisher's method is ABO.  In this paper we will use the
Bahadur relative efficiency as our primary measure of comparing p-value combination methods. Assuming two statistical tests are formed to test the same hypothesis
and have exact slopes $c_1(\theta)$ and $c_2(\theta)$ respectively, then the
ratio $c_1(\theta)/c_2(\theta)$ is the exact Bahadur efficiency of test $1$
relative to test $2$, and $c_1(\theta)/c_2(\theta)>1$ implies that test 1 is asymptotically more efficient than test 2.
\begin{lemma}
  For $\theta=0$, $-\frac{2}{n}\log(p)\rightarrow 0$ with probability one, i.e., if the effect size is $0$, the exact slope $c(\theta)$ of the statistical test is $0$.
\end{lemma}
\begin{proof}
For $\theta=0$, $-2log(p) \sim \chi^2_2$  since the p-value $p$ is distributed uniformly in $(0,1)$.
Since$-2log(p)$ is tight and $\frac{1}{n} \rightarrow 0$ with probability one, then  we have $-\frac{2}{n}\log(p)\rightarrow 0$
i.e., $c(0)=0$.
\end{proof}

\begin{lemma}
  Let $F_{\chi_k}(x)=P(\chi_k\leq x)$, where $\chi_k^2$ follows chi-square distribution with degree of freedom k. Then $log\left(1-F_{\chi_k}(x)\right)\rightarrow -\frac{1}{2}x^2(1+o(1))\;as\; x \rightarrow \infty$.
\end{lemma}
\begin{proof}
The proof is given by \citet{bahadur1960stochastic} in page 283.\\~\\
\end{proof}

\citet{littell1971asymptotic} showed that given $K$ independent studies with
p-values, sample sizes and exact slopes $p_k,n_k,c_k(\theta), k=1, \dots, K$
respectively, the exact slope of Fisher's method is
$c_{Fisher}(\theta)=\sum_{k=1}^K \lambda_kc_k(\theta)$. Let $c_{AW}(\theta)$ be exact slope from AW-Fisher's method. Since the Fisher's method is ABO, i.e., $c_{Fisher}(\theta)$ is the
largest among all p-value combination procedures \citep{littell1973asymptotic}, under
the assumption $\theta_k\equiv\theta\neq 0$, 
to prove the AW-Fisher's method is also ABO, here we show that the exact slopes from AW-Fisher and Fisher's methods are the same.

\begin{theorem}
Under the conditions about exact slopes and proportion of sample sizes, we have $c_{AW}(\theta)=c_{Fisher}(\theta)$, i.e., the AW-Fisher's method is ABO.
\end{theorem}

\begin{proof}
Let $\{p'_1, \dots, p'_j\}$ be a subset of $\{p_1, \dots, p_K\}$ with size
$j$ for $j=1,\dots,K$, denote $L_{obs}=\displaystyle\min_{\vec {\textbf{w}}} L(T(\vec{\textbf{w}}; \vec{\textbf{p}}_{obs}))$,then for a given test statistic $s_{obs}$,
\[
\begin{split}
P(S \ge  s_{obs}) &= P( \min_{\vec {\textbf{w}}} L(T(\vec{\textbf{w}}; \vec{\textbf{P}})) \le  L_{obs} )\\
		& =1- P( \min_{\vec {\textbf{w}}} L(T(\vec{\textbf{w}}; \vec{\textbf{P}})) \ge  L_{obs} ) \\
		& = 1-P( \bigcap_{j = 1}^K \bigcap_{{\textbf{w}}\in\Omega_j}
		L(T(\vec{\textbf{w}}; \vec{\textbf{P}})) \ge  L_{obs} ) \\
\text{Then by Bonferroni's inequality, we have:}\\
		&\le 1-\left(1-\sum_{j = 1}^K P(  \bigcup_{{\textbf{w}}\in\Omega_j}  L(T(\vec{\textbf{w}}; \vec{\textbf{P}})) \le L_{obs})\right)\\
                   &= \sum_{j = 1}^K
                    P\left(\bigcup_{{\textbf{w}}\in\Omega_j}\left\{-2\sum_{i=1}^j \log(P'_i) \ge \chi_{2j}^{-2}(L_{obs})\right\}\right) \\
                    &\le \sum_{j = 1}^K\sum_{{\textbf{w}}\in\Omega_j}P\left(-2\sum_{i=1}^j \log(P'_i) \ge \chi_{2j}^{-2}(L_{obs})\right)\\
                   & = (2^{K}-1)L_{obs} \\
\end{split}
\]
as $n$ goes to infinity.\\
Since the adaptive weights are consistent,and by Lemma 1 and 2, we have
\[
\begin{split}
\lim_{n \rightarrow \infty} -\frac{2}{n}\log(L_{obs}) & = \lim_{n \rightarrow \infty} -\frac{2}{n}\log(1-\chi_{2d(\hat{\textbf{w}})}^{2}(-2\sum_{i=1}^K \hat{w}_i\ log(p_i)))\\
&=\lim_{n \rightarrow \infty} -\frac{2}{n}\log\left(1-F_{\chi_{2d(\hat{\textbf{w}})}}\left(\sqrt{-2\sum_{i=1}^K \hat{w}_i\ log(p_i)}\right)\right) \\ &=-\frac{2}{n}\left\{\left(\sum_{i=1}^K \hat{w}_i\ log(p_i)\right)(1+o(1))\right\} \\
& \rightarrow \sum_{i=1}^K w^*_i \lambda_i c_i(\theta)
\end{split}
\]
So
\[
\lim_{n \rightarrow \infty} -\frac{2}{n} \log\left(P(S \ge s_{obs})\right) \geq \lim_{n \rightarrow \infty} -\frac{2}{n} \{ L_{obs} + \log(2^{K}-1)\} = \sum_{i=1}^K w^*_i \lambda_i c_i(\theta)
\]
On the other hand, since
\[
\begin{split}
P(S \ge s_{obs}) = &  P( \min_{\vec {\textbf{w}}} L(T(\vec{\textbf{w}}; \vec{\textbf{P}})) \le  L_{obs} )\\
=& 1-P( \min_{\vec {\textbf{w}}} L(T(\vec{\textbf{w}}; \vec{\textbf{P}})) \ge  L_{obs} )\\
\ge & 1-P(  L(T(\vec{\textbf{w}}; \vec{\textbf{P}})) \ge  L_{obs} )\\
= & P(\{-2\sum_{i=1}^j  \log(P'_i) \ge \chi_{2j}^{-2}(L_{obs})\}) \\
= & L_{obs} \\
\end{split}
\]
by utilizing the consistency of the adaptive weights and Lemma 3 and 4 again, we have
\[
\lim_{n \rightarrow \infty} -\frac{2}{n} \log(P(S \ge s_{obs})) \le \lim_{n \rightarrow \infty} -\frac{2}{n}\log(L_{obs})
 \rightarrow  \sum_{i=1}^K w^*_i \lambda_i c_i(\theta)
\]
Therefore $c_{AW}(\theta) \rightarrow \sum_{i=1}^K w^*_i \lambda_i c_i(\theta) $.
Other the other hand, $c_{Fisher}(\theta) = \sum_{i=1}^K \lambda_i c_i(\theta) = \sum_{i=1}^K w^*_i \lambda_i c_i(\theta)$.
Therefore, $c_{AW}(\theta)=c_{Fisher}(\theta)$ and so AW-Fisher's method is also ABO.
\end{proof}
\section{Simulations}


In this simulation, we use numerical approach to evaluate and verify the convergence rate of the weight estimates. We consider $K=4$ studies with the same sample size $n$ for all studies. The data are generated from standard normal distribution with $n/2$ from control and $n/2$ from treatment groups. We first generate independent identical random variable $X_{j1i}, X_{j2i} \sim N(0, 1), i = 1, \ldots, n$ if study $j$ has no treatment effect and $X_{2i} \sim N(\mu_i, 1), i = 1, \ldots, n$
if study $j$ with treatment effect. Let $\Delta_{\bar X_i} = \bar X_{j1} - \bar X_2$, then , $\Delta_{\bar X} \sim N(0, 4/n)$ under the null hypothesis and $\Delta_{\bar X} \sim N(\mu, 4/n)$ under the alternative hypothesis, where $\bar{X_j} = \frac{2}{n}\sum_i X_{ji}, j=1,2, i = 1,\dots,n/2$.
Therefore, p-value for study $j$ can be calculated as
\begin{equation}
p_j = 1-2\Phi(-|\Delta_{\bar X_j}|).
\label{eq:pvalues}
\end{equation}
where $\Phi$ is the cumulative density function of standard normal distribution.

In the first simulation, we estimate the convergence rate for studies with non-zero effective size, i.e. $P(\hat{w}_{k} = 0 | w^*_{k} = 1)$. We have shown in Section~\ref{sec:Properties} that
\[
\lim_{n \rightarrow \infty} P(\hat{w}_{k} = 0 | w^*_{k} = 1)
\rightarrow  n \exp\left\{ - \frac{c_{k}\lambda_{k}n}{2}\right\}.
\]
We evaluate this probability based on 1 million simulations for each $n$ from 200 to 1000, in which all 4 studies are considered with effect sizes $0.2,0.3,0.4,0.5$ for studies 1 to 4.

Given sample size $n$, 
the AW weight estimate can be obtained according to equation~(\ref{eq:weight}).
We generated scattered plot of $\hat P(\hat{w}_k=0 | w^*_k = 1) / n$ with respect to sample size and it is shown in Figure~\ref{fig:AWaccuracy}(a).
We further fitted a curve with functional form $ a n \exp\left\{ - bn\right\}$, where $a$ and $b$ are the parameters estimated from the data. The estimates are $\hat a = $ and $\hat b= $. The fitted curve agrees with the functional form $ a ~n \exp\left\{ - bn\right\}$ very well.

In the second simulation, we set the one study (study 4) with effective size 0 and all other studies with effect size 0.4 to estimate $P(\hat{w}_{k} = 1 | w^*_{k} = 0)$. As shown in Section~\ref{sec:Properties}, the converge rate for study with zero effect size is
\[
\begin{split}
\lim_{n \rightarrow \infty} P(\hat{w}_{k} = 1 | w^*_{k} = 0)
\rightarrow   O(1/n).
\end{split}
\]
We again use 1 million simulations to estimate this probability for $n=200 - 1000$ for study 4. And the AW weight estimate can be obtained according to equation~(\ref{eq:weight}).
The scattered plot of $P(\hat{w}_k=1 | w^*_k = 0)$ against sample size is shown in Figure~\ref{fig:AWaccuracy}(b). We further fitted a curve with functional form $ \frac{1}{a + bn}$, where $a$ and $b$ are the parameters estimated from the data and the fitted curve agrees with the simulated data very well.

\begin{figure}[htbp]
	\caption{Comparing accuracy of the new approach and permutation approach to obtain the AW p-values.
	The scattered plots are p-values from the two methods against the closed-form solution.}
	\label{fig:AWaccuracy}
	\centering
	\subfigure[$P(\hat{w_k}=0 | w^*_k = 1)$]{
		\includegraphics[height=0.45\columnwidth]{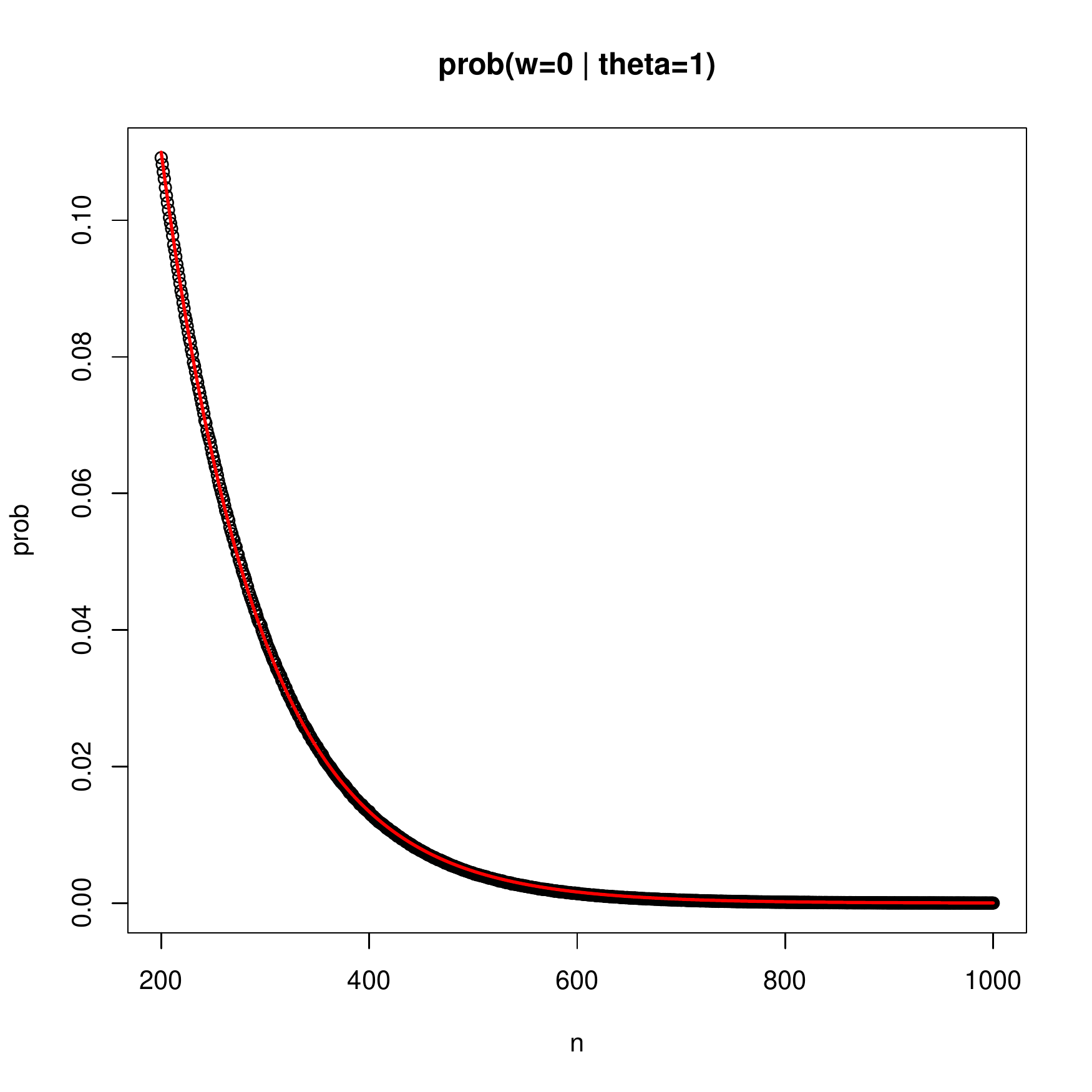}
		\label{fig:theta1w0}
	}
	\subfigure[$P(\hat{w_k}=1 | w^*_k = 0)$]{
		\includegraphics[height=0.45\columnwidth]{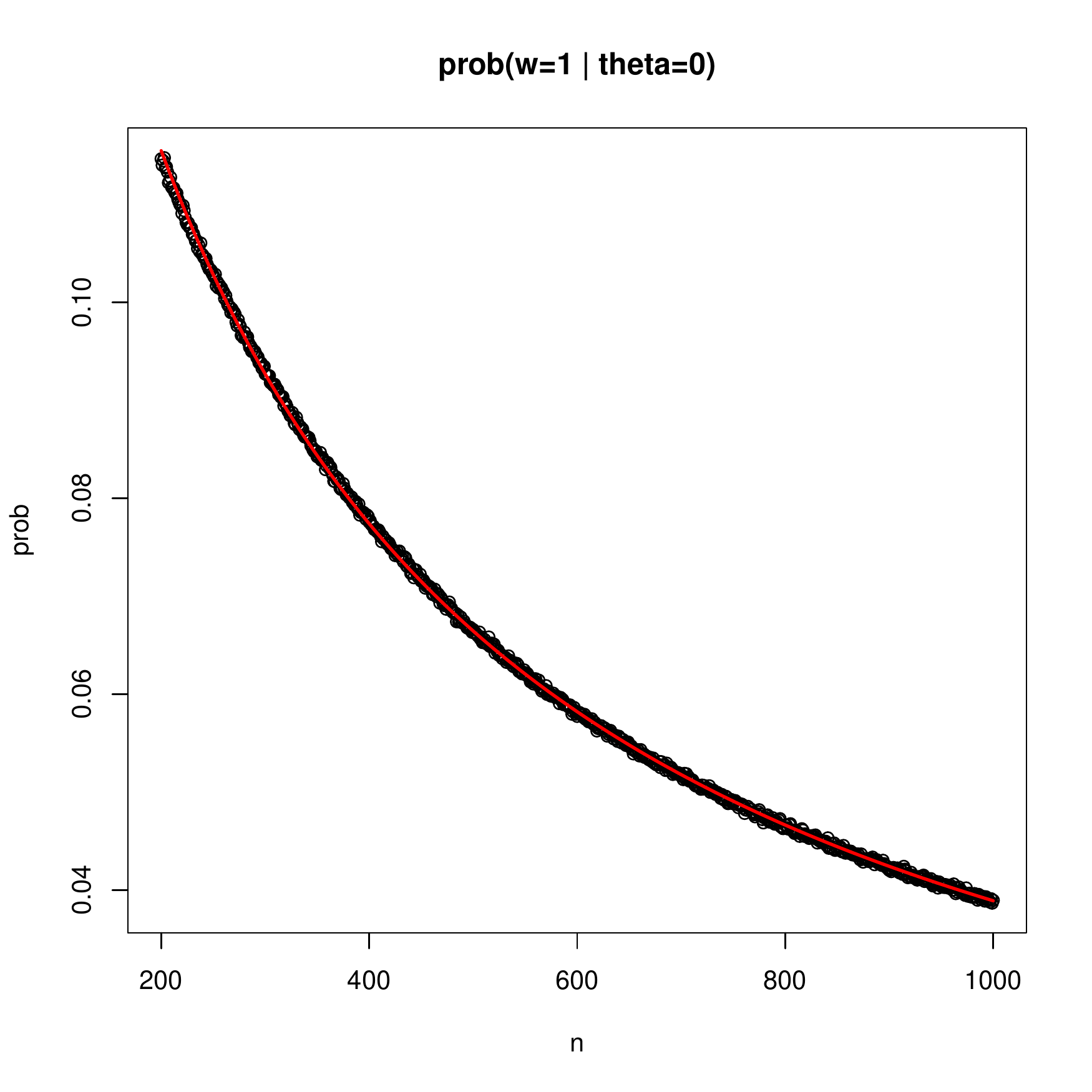}
		\label{fig:theta0w1}
	}
\end{figure}

\section{Conclusion}

The AW-Fisher's method proposed in \citet{Li2011adaptively} has shown to have many good properties such as admissibility and better overall power compared to min-P,
max-P and Fisher's methods in various situations. More importantly, the adaptive weights can provide additional information about heterogeneity of effect sizes in the different studies, a feature particularly appealing in the genomic meta-analysis.

For the practical usage of adaptive weights, such as uncertainty of adaptive weight estimates, has been discussed in Huo et. al (2017). A fast algorithm to estimate accurate p-values based on importance sampling also proposed in Huo et. al (2017).

In this paper, we further studied the asymptotic properties of AW-Fisher's method. We have shown the consistency of adaptive weights and asymptotic Bahadur optimality to reaffirm the validity and value of AW-Fisher's method. The asymptotic convergence rate of AW weight has been verified using simulation.

\bibliographystyle{biom}
\bibliography{references}
\label{lastpage}

\end{document}